\documentclass[10pt]{amsart}
\usepackage{amsmath,a4wide}
\usepackage{amssymb}
\usepackage{amsthm}
\usepackage{graphicx}

 \newtheorem{theorem}{Theorem}[section]
 
 \newtheorem{proposition}[theorem]{Proposition}
 \newtheorem{lemma}[theorem]{Lemma}

 {\theoremstyle{definition}
 \newtheorem{remark}[theorem]{Remark}}
 {\theoremstyle{definition}
  \newtheorem{definition}[theorem]{Definition}}
 {\theoremstyle{definition}
 \newtheorem{example}[theorem]{Example}} 
 {\theoremstyle{definition}
 }

\numberwithin{equation}{section}
 \newcommand{\LO}{\textrm{LO}}
 \newcommand{\Z}{\mathbb{Z}}

 \newcommand{\mF}{\mathcal{F}}
 
 \newcommand{\Conv}{\mathsf{Conv}}

\title{Isolated orderings on amalgamated free products}
\author{Tetsuya Ito}
\address{Research Institute for Mathematical Sciences, Kyoto university
Kyoto, 606-8502, Japan}
\email{tetitoh@kurims.kyoto-u.ac.jp}
\urladdr{http://www.kurims.kyoto-u.ac.jp/~tetitoh/}

\begin{document}
\maketitle

\begin{abstract}
We show that an amalgamated free product $G*_{A}H$ admits a discrete isolated ordering, under some assumptions of $G,H$ and $A$. This generalizes the author's previous construction of isolated orderings, and unlike known constructions of isolated orderings, can produce an isolated ordering with many non-trivial proper convex subgroups.
\end{abstract}

\section{Introduction}

A total ordering $<_{G}$ of a group $G$ is a \emph{left-ordering} if the relation $<_{G}$ is preserved by the left action of $G$ itself, namely, $a<_{G} b$ implies $ga <_{G} gb$ for all $a,b,g \in G$. A group admitting a left-ordering is called \emph{left-orderable}.

For $g \in G$, let $U_{g}$ be the set of left-orderings $<_{G}$ of $G$ that satisfy $1<_{G} g$. The set of all left-orderings of $G$ can be equipped with a topology so that $\{U_{g}\}_{g \in G}$ is an open sub-basis. We denote the resulting topological space by $\LO(G)$ and call the \emph{space of left-orderings of $G$} \cite{si}.

An \emph{isolated ordering} is a left ordering which is an isolated point in $\LO(G)$. A left-ordering $<_{G}$ is isolated if and only if $<_{G}$ is determined by the sign of finitely many elements. That is, $<_{G}$ is isolated if and only if there exists a finite subset $\{g_1,\ldots,g_{n}\}$ of $G$ such that $\bigcap_{i=1}^{n} U_{g_i} =\{<_{G}\}$. We call such a finite subset a \emph{characteristic positive set} of $<_{G}$. 
In particular, if the positive cone $P(<_{G})$ of a left ordering $<_{G}$, the sub semi-group of $G$ consisting of $<_{G}$-positive elements, is finitely generated then $<_{G}$ is isolated.

An isolated ordering $<_{G}$ of $G$ is \emph{genuine} if $\LO(G)$ contains non-isolated points. This is equivalent to saying that $\LO(G)$ is not a finite set. Since the classification of groups admitting only finitely many left-orderings (non-genuine isolated orderings) is known (see \cite[Theorem 5.2.1]{km}), we concentrate our attention to genuine isolated orderings. 
Several classes of groups do not have genuine isolated ordering.
The non-existence of isolated orderings are observed for the free abelian groups of rank $>1$ \cite{si} and the free groups of rank $>1$ \cite{n1}. More generally, by using a dynamical realizations, it is shown that the free products of more than one groups \cite{ri}, and virtually solvable groups \cite{rt} never admit a genuine isolated ordering.

Recent developments provide several examples of genuine isolated orderings, but our catalogues and knowledge are still limited and it is still hard to predict when a left-orderable group admits an isolated ordering. At present, we have three ways of constructing (genuine) isolated orderings; Dehornoy-like orderings \cite{i1,n2}, partially central cyclic amalgamation \cite{i2}, and triangular presentations with certain special elements \cite{de}. 

The aim of this paper is to extend a partially central cyclic amalgamated product construction of isolated orderings \cite{i2} in more general and abstract settings. 

To state main theorem, we introduce the following two notions.
Let $A$ be a subgroup of a left-orderable group $G$.
First we extend the notion of isolatedness in a relative setting.

\begin{definition} 
Let $\textsf{Res}:LO(G)\rightarrow LO(A)$ be the continuous map induced by the restriction of left orderings of $G$ on $A$. We say a left ordering $<_{G}$ of $G$  is \emph{relatively isolated} with respect to $A$ if $<_{G}$ is an isolated point in the subspace $\textsf{Res}^{-1}(\textsf{Res} (<_{G})) \subset \LO(G)$. Thus, $<_{G}$ is relatively isolated if and only if there exists a finite subset $\{g_1,\ldots,g_{n}\}$ of $G$ such that $\textsf{Res}^{-1}(\textsf{Res} (<_{G})) \cap \bigcap U_{g_i} =\{<_{G}\}$. We say such a finite set a \emph{characteristic positive set} of $<_{G}$ relative to $A$.
\end{definition}

The next property plays a crucial role in our construction of isolated orderings. 

\begin{definition}
We say that a subgroup $A$ is a \emph{stepping} with respect to a left-ordering $<_{G}$ of $G$ if for each $g \in G$ both the maximal and the minimal
\begin{gather*}
\begin{cases}
a(g)= \max_{<_{G}}\{a \in A \: | \: a\leq_{G} g\} \\  
a_{+}(g)= \min_{<_{G}}\{a \in A \: | \: g <_{G} a\}.
\end{cases}
\end{gather*}
always exist. 
\end{definition}

Using these notions our main theorem is stated as follows.
Here is a situation we consider. Let $A$, $G$ and $H$ be left-orderable groups. 
We fix embeddings $i_{G}: A \hookrightarrow G$ and $i_{H}: A \hookrightarrow H$ so we always regard $A$ as a common subgroup of $G$ and $H$.

\begin{theorem}
\label{theorem:main}
Let $<_{G}$ and $<_{H}$ be discrete orderings of $G$ and $H$. 
Assume that $<_{G}$ and $<_{H}$ satisfy the following conditions.
\begin{enumerate}
\item[(a)] The restriction of $<_{G}$ and $<_{H}$ on $A$ yields the same left ordering $<_{A}$ of $A$.
\item[(b)] $A$ is a stepping with respect to both $<_{G}$ and $<_{H}$.
\item[(c)] $<_{G}$ is isolated and $<_{H}$ is relatively isolated with respect to $A$.
\end{enumerate}
Then the amalgamated free product $X=G *_{A} H$ admits isolated orderings $<_{X}^{(1)}$ and $<_{X}^{(2)}$ which have the following properties:

\begin{enumerate}
\item Both $<_{X}^{(1)}$ and $<_{X}^{(2)}$ extend the orderings $<_{G}$ and $<_{H}$: if $g <_{G} g'$ $(g,g' \in G)$ then $ g <_{X}^{(i)} g'$, and if $h <_{H} h'$ $(h,h' \in H)$ then $ h<_{X}^{(i)} h'$ $(i=1,2)$. 
\item If $\{g_{1},\ldots,g_m\}$ is a characteristic positive set of $<_{G}$ and $\{h_{1},\ldots,h_n\}$ is a characteristic positive set of $<_{H}$ relative to $A$, then 
\[\{g_1,\ldots,g_m, h_1,\ldots,h_n, h_{\sf min}a_{\sf min}^{-1}g_{\sf min}\}\]
 is a characteristic positive set of $<_{X}^{(1)}$ and
\[\{g_1,\ldots,g_m, h_1,\ldots,h_n, g_{\sf min}a_{\sf min}^{-1}h_{\sf min}\} \]
is a characteristic positive set of $<_{X}^{(2)}$.
 Here $a_{\sf min}$, $g_{\sf min}$ and $h_{\sf min}$ represent the minimal positive elements of the orderings $<_{A}$, $<_{G}$ and $<_{H}$, respectively. 
(Note that $A$ is a stepping implies that $<_{A}$ is discrete, see Lemma \ref{lemma:discA}).
\item $<_{X}^{(1)}$ is discrete with the minimal positive element $h_{\sf min}a_{\sf min}^{-1}g_{\sf min}$, and $<_{X}^{(2)}$ is discrete with the minimal positive element $g_{\sf min}a_{\sf min}^{-1}h_{\sf min}$.
\item $A$ is a stepping with respect to the orderings $<_{X}^{(1)}$ and $<_{X}^{(2)}$.
\end{enumerate}

\end{theorem}

The assumption (a) is an obvious requirement for $X$ to have a left ordering  extending both $<_{G}$ and $<_{H}$. The crucial assumptions are (b) and (c). It should be emphasized that the orderings $<_{A}$ and $<_{H}$ may not be isolated. 

We also note that, The property (4) allows us to iterate a similar construction, hence Theorem \ref{theorem:main} produces huge examples of isolated orderings.

\begin{remark}
As for the existence of isolated orderings, Theorem \ref{theorem:main} contains the main theorem of \cite{i2} as a special case, but \cite[Theorem 1.1]{i2} contains much stronger results.

In \cite{i2}, we treated the case that $A=\Z$ with additional assumptions that the isolated ordering $<_{H}$ is preserved by the right action of $A$, and that $A$ is central in $G$. Under these assumptions, we proved that the positive cone of the resulting isolated ordering is \emph{finitely generated}, and determined all convex subgroups. Moreover, one can algorithmically determine whether $x<_{X} x'$ or not.

On the other hand, for the isolated orderings $<_{X}^{(i)}$ in Theorem \ref{theorem:main}, we do not know whether its positive cone is finitely generated or not in general, and a computation of $<_{X}^{(i)}$ is more complicated. As for the computational issues, see Remark \ref{rem:computation}.
\end{remark}

In light of the above remark, finding a generating set of the positive cone of $<_{X}^{(i)}$, and determining when it is finitely generated are quite interesting.

As for convex subgroups, in Proposition \ref{prop:convex} we show that a  convex subgroup of $A$ with additional properties yields a convex subgroup of $(X,<_{X})$. Thus, the resulting isolated ordering of $X$ can admit many non-trivial convex subgroups. This also makes a sharp contrast in \cite{i2}, where the obtained isolated ordering contains exactly one non-trivial proper convex subgroup. 
It should be emphasized that the Dubrovina-Dubrovin ordering of the braid groups \cite{dd} are the only known examples of genuine isolated ordering with more than one proper non-trivial convex subgroup. In Example \ref{exam:convex}, starting from $\Z$ with standard ordering, the simplest isolated ordering, we construct many isolated orderings with more than one non-trivial convex subgroups.

\section*{Acknowledgement}

The author was partially supported by the Grant-in-Aid for Research Activity start-up, Grant Number 25887030.

\section{Construction of isolated orderings}

Let $(S,<_{S})$ be a totally ordered set. For $s,s' \in S$, we say $s'$ is the \emph{successor} of $s$ and denote by $s \prec_{S} s'$, if $s'$ is the minimal element in $S$ that is strictly greater than $s$ with respect to the ordering $<_{S}$.

A left ordering $<_{G}$ of a group $G$ is \emph{discrete} if there exists the successor $g_{\sf min}$ of the identity element. That is, $<_{G}$ admits the minimal $<_{G}$-positive element. By left-invariance, a discrete left ordering $<_{G}$ satisfy 
$g g_{\sf min}^{-1} \prec_{G} g \prec_{G} gg_{\sf min}$ for all $g \in G$.

Let us consider the situation in Theorem \ref{theorem:main}: Let $G$ and $H$ be groups admitting discrete left orderings $<_{G}$ and $<_{H}$, and $A$ be a common subgroup of $G$ and $H$, such that the restriction of $<_{G}$ and $<_{H}$ yield the same left ordering $<_{A}$ .

The assumption that $A$ is a stepping (assumption (b)) implies the following.
\begin{lemma}
\label{lemma:discA}
For a subgroup $A$ of a left-orderable group $G$, if $A$ is a stepping with respect to a left-ordering $<_{G}$, then the restriction of $<_{G}$ on $A$ is discrete.
\end{lemma}
\begin{proof}
From the definition of stepping, 
\[ a_{\sf min} = \min_{<_{A}} \{a \in A \: | \: 1 <_{A} a\}=\min_{<_{G}} \{a \in A \: | \: 1 <_{G} a\}  = a_{+}(1) \]
exists. 
\end{proof}

Thus $<_{A}$ is also discrete. We denote the minimal positive elements of $<_{A}$, $<_{G}$ and $<_{H}$ by $a_{\sf min}$, $g_{\sf min}$ and $h_{\sf min}$, respectively. We put $g_{\sf M} = a_{\sf min}g_{\sf min}^{-1}$ and $h_{\sf M} = a_{\sf min} h_{\sf min}^{-1}$, so $g_{\sf M} \prec_{G} a_{\sf min}$ and $h_{M} \prec_{H} a_{\sf min}$.

We start to construct an isolated ordering on a group $X=G *_{A} H$. 
We explain the construction of the isolated ordering $<_{X}^{(1)}$, which we simply denote by $<_{X}$. The construction of $<_{X}^{(2)}$ is similar. We just apply the same construction by interchanging the role of $G$ and $H$.

The amalgamated free product structure of $X$ induces a filtration 
\[ \mF_{-1}(X) \subset \mF_{-0.5}(X) \subset \mF_{0}(X) \subset \mF_{0.5}(X) \subset \mF_{1}(X) \subset \mF_{2}(X)\subset \cdots \subset \mF_{i}(X) \subset \mF_{i+1}(X) \subset \cdots\]
defined by
\begin{gather*}
\begin{cases}
\mF_{-1}(X)=\emptyset, \;\;\; \mF_{-0.5}(X)=A, \;\;\; \mF_{0}(X) = H, \;\;\; \mF_{0.5}(X)= G \cup H, \\
\mF_{2i+1}(X) = G \mF_{2i},\\
\mF_{2i}(X) = H \mF_{2i-1}.
\end{cases}
\end{gather*}

The non-integer parts of the filtrations are exceptional, and the filtration $\mF_{0.5}(X)$ is the most important because it is the restriction on $\mF_{0.5}(X)$ that eventually characterizes the isolated ordering $<_{X}$.

Starting from $<_{G}$ and $<_{H}$, we inductively construct a total ordering $<_{i}$ on $\mF_{i}(X)$. To be able to extend $<_{i}$ to a left ordering of $X$, we need the following obvious property.

\begin{definition}
We say a total ordering $<_{i}$ on $\mF_{i}(X)$ is \emph{compatible} if for any $x \in X$ and $s,t \in \mF_{i}(X)$, $xs <_{i} xt$ whenever $s<_{i} t$ and $xs, xt \in \mF_{i}(X)$. 
\end{definition}

By definition, if $<_{i}$ is a restriction of a left ordering of $X$ on $\mF_{i}(X)$, then $<_{i}$ is compatible. Conversely, Bludov-Glass proved that a compatible ordering $<_{i}$ on $\mF_{i}(X)$ can be extended to a compatible ordering $<_{i+1}$ of $\mF_{i+1}(X)$ under some conditions \cite{bg}. This is a crucial ingredient of the proof of Bludov-Glass' theorem on necessary and sufficient conditions for an amalgamated free product to be left-orderable \cite[Theorem A]{bg}.

From the point of view of the topology of $\LO(G*_{A}H)$, it is suggestive to note that Bludov-Glass' extension of $<_{i}$ to $<_{i+1}$ is far from unique. This illustrates and explains the intuitively obvious fact that ``most'' left orderings of $G*_{A} H$ are not isolated. Our isolated ordering is constructed by specifying a situation that Bludov-Glass' extension procedure must be unique.

As the first step of construction, we define an ordering $<_{\sf base}$ on $\mF_{0.5}(X)$. Since we have assumed that $A$ is a stepping with respect to both $<_{G}$ and $<_{H}$, we have the function
\[ a: \mF_{0.5}(X) \rightarrow A \]
defined by 
\begin{gather}
\label{eqn:a}
a(x) = 
\begin{cases}
\max_{<_{G}} \{a \in A \: | \: a \leq_{G} x\} & (x \in G) \\
\max_{<_{H}} \{a \in A \: | \: a \leq_{H} x\} & (x \in H).
\end{cases}
\end{gather}
 
 Using the function $a$, we define the total ordering $<_{\sf base}$ as follows:
\begin{gather}
\label{eqn:base}
\begin{cases}
g <_{\sf base} g' & \text{ if } g, g' \in G \text{ and } g<_{G}g' \\
h <_{\sf base} h' & \text{ if } h, h' \in H \text{ and } h<_{H}h' \\
h <_{\sf base} g & \text{ if } h \in H-A, g \in G-A \text{ and } a(h) \leq_{A} a(g) \\
g <_{\sf base} h & \text{ if } h \in H-A, g \in G-A \text{ and } a(g) <_{A} a(h) \\
\end{cases}
\end{gather} 

The ordering $<_{\sf base}$ can be schematically understood by Figure \ref{fig:base}.

\begin{figure}[htbp]
 \begin{center}
\includegraphics*[width=120mm]{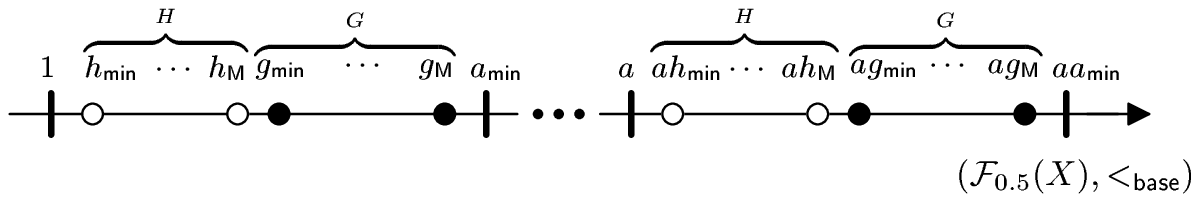}
\caption{Ordering $<_{\sf base}$ on $\mF_{0.5}(X)$.}
\label{fig:base}
\end{center}
\end{figure}

\begin{lemma}
\label{lemma:base}
The ordering $<_{\sf base}$ is the unique compatible ordering of $\mF_{0.5}(X)$ such that
\begin{description}
\item[B1] The restriction of $<_{\sf base}$ on $G$ and $H$ agrees with $<_{G}$ and $<_{H}$, respectively.
\item[B2] $h_{\sf M}=a_{\sf min}h_{\sf min}^{-1} <_{\sf base} g_{\sf min}$.
\end{description}
\end{lemma}
\begin{proof}
By definition, $<_{\sf base}$ is a compatible ordering with {\bf [B1]} and {\bf [B2]}.
Assume that $<'$ is another compatible total ordering on $\mF_{0.5}(X)$ with the same properties. To see the uniqueness, it is sufficient to show that for $g \in G-A$ and $h \in H-A$, $h<_{\sf base} g$ implies $h <' g$.

By definition of $<_{\sf base}$, $a(h) \leq_{A} a(g)$. If $a(h) <_{A} a(g)$, then $h <' a(h) a_{\sf min} \leq' a(g) <' g$ so $h<'g$. Assume that $a(h)=a(g)$ and put $a=a(g)=a(h)$. By {\bf [B1]}, $1 <' a^{-1}h <' a_{\sf min}$ hence $1 <' a^{-1} h \leq' h_{\sf M} = a_{\sf min}h_{\sf min}^{-1}$.
Similarly, $1 <' a^{-1}g$ so $g_{\sf min} \leq' a^{-1}g$.
By {\bf [B2]},
\[ a^{-1} h \leq' h_{\sf M}<' g_{\sf min} \leq' a^{-1}g, \]
 hence $ a^{-1} h <' a^{-1} g $. Since $<'$ is compatible, $h <' g$.
\end{proof}

Lemma \ref{lemma:base}, combined with our assumption (c) of Theorem \ref{theorem:main}, shows the following.

\begin{proposition}
\label{prop:charbase}
The compatible ordering $<_{\sf base}$ is characterized by finitely many inequalities: Let $\{g_{1},\ldots,g_m\}$ be a characteristic positive set of $<_{G}$ and $\{h_{1},\ldots,h_n\}$ be a characteristic positive set of $<_{H}$ relative to $A$. Then $<_{\sf base}$ is the unique compatible ordering on $\mF_{0.5}(X)$ that satisfies the inequalities
\begin{gather}
\label{eqn:char}
\begin{cases}
1<_{\sf base} g_{i} & (i=1,\ldots,m),\\
1 <_{\sf base} h_{j} & (j=1,\ldots,n), \\
a_{\sf min}h_{\sf min}^{-1} <_{\sf base} g_{\sf min}. 
\end{cases}
\end{gather}
\end{proposition}
\begin{proof}
The set of inequalities $\{1<_{\sf base} g_{i}\}$ uniquely determines the restriction of $<_{\sf base}$ on $G$ so in particular, determines the restriction of $<_{\sf base}$ on $A$. Since $<_{H}$ is relatively isolated with respect to $<_{H}$, the additional inequalities $\{1<_{\sf base} h_{i}\}$ uniquely determines the restriction of $<_{\sf base}$ on $H$. Therefore the family of inequalities (\ref{eqn:char}) implies {\bf [B1]} and {\bf [B2]} in Lemma \ref{lemma:base}.
\end{proof}

The next step is to extend the ordering $<_{\sf base}$ to a compatible ordering $<_{1}$ of $\mF_{1}(X)=GH$. For $a \in A$, let 
\begin{eqnarray*}
\Delta_{a} &= & \{h \in H-A \: | \: a(h)=a\} \\
& = & \{h \in H-A \: | \: a <_{H} h <_{H} aa_{\sf min} \} = \{h \in H-A \: | \: ah_{\sf min} \leq_{H} h \leq_{H} ah_{\sf M}\}.
\end{eqnarray*}

First we observe the following property which plays a crucial role in proving the uniqueness.

\begin{lemma}
\label{lemma:key0}
For $g,g' \in G$ and $h,h' \in H$, if $ga(h)=g'a(h')$ then  $g\Delta_{a(h)}=g'\Delta_{a(h')}$.
\end{lemma}
\begin{proof}
$ga(h)=g'a(h')$ implies that $g^{-1}g' = a(h)a(h')^{-1} \in A$.
This shows $(g^{-1}g') \Delta_{a(h')} = a(h)a(h')^{-1}\Delta_{a(h')} = \Delta_{a(h)}$ hence $g\Delta_{a(h)}=g'\Delta_{a(h')}$.
\end{proof}

\begin{proposition}
\label{prop:extbto1}
There exists a unique compatible total ordering $<_{1}$ on $\mF_{1}(X)$ that extends $<_{\sf base}$.
\end{proposition}
\begin{proof}

For each $a \in A$ and $g\in G-A$, we regard $g\Delta_{a}$ as a totally ordered set equipped with an ordering $<_{1}$ defined by $gh <_{1} gh'$ ($h,h' \in \Delta_{a}$) if and only if $h <_{H} h'$.

First we check that this ordering $<_{1}$ is well-defined on each $g\Delta_{a}$.  Assume that $g\Delta_{a} = g' \Delta_{a'}$ as a subset of $\mF_{1}(X)$. Let $gh_{0}= g'h'_{0}, gh_{1}=g'h'_{1}$ be elements of $g\Delta_{a} = g' \Delta_{a'}$, where $h_{0},h_{1} \in \Delta_{a}$ and $h'_{0}, h'_{1} \in \Delta_{a'}$.
Note that $g\Delta_{a} = g' \Delta_{a'}$ implies that $g^{-1}g' \in A$. Therefore, 
\begin{eqnarray*}
gh_{0}<_{1} gh_{1} & \iff & h_{0}<_{H} h_{1} \\
& \iff & (g^{-1}g')h'_{0} <_{H} (g^{-1}g')h'_{1} \\
& \iff & h'_{0} <_{H} h'_{1}\\
& \iff & g'h'_{0} <_{1} g'h'_{1}.
\end{eqnarray*}
This shows that $<_{1}$ is a well-defined total ordering on $g\Delta_{a}$.

Since $\mF_{1}(X) = \mF_{0}(X) \cup \left(\bigcup g \Delta_{a} \right)$, we construct the desired ordering $<_{1}$ by inserting the ordered sets $g\Delta_{a}$ into $\mF_{0}(X)$. We show that the way to inserting $g\Delta_{a}$ is unique. 

First of all, $a <_{\sf base} h <_{\sf base} ag_{\sf min}$ for $h \in \Delta_{a}$, so a compatible ordering $<_{1}$ must satisfy 
\[ ga <_{1} gh <_{1} gag_{\sf min}  \;\;\; (g \in G-A).\]
By definition of $<_{\sf base}$, $ga \prec_{\sf base} gag_{\sf min}$, that is, there are no elements of $\mF_{0.5}(X)$ that lies between $ga$ and $gag_{\sf min}$.
This says that to get a compatible ordering, we must insert the ordered set $g\Delta_{a}$ between $ga$ and $gag_{\sf min}$. Moreover, by Lemma \ref{lemma:key0}, $ga(h) = g'a(h')$ implies $g \Delta_{a(h)}=g'\Delta_{a(h')}$. This means that the ordered set $g\Delta_a$ inserted between $ga$ and $gag_{\sf min}$ must be unique.

\begin{figure}[htbp]
 \begin{center}
\includegraphics*[width=100mm]{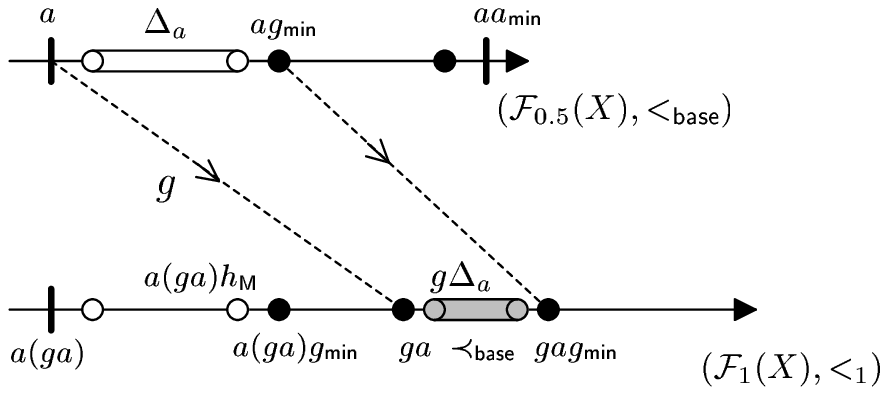}
\caption{Ordering $<_{1}$: Inserting $g\Delta_{a}$ between $ga$ and $gag_{\sf min}$.}
\label{fig:o1}
\end{center}
\end{figure}

Therefore there is the unique way of inserting $g\Delta_{a}$ into $\mF_{0}(X)$ to get a compatible ordering on $\mF_{1}(X)$. The process of inserting $g\Delta_{a}$ is schematically explained in Figure \ref{fig:o1}.

The resulting ordering $<_{1}$ is written as follows. For $x = gh$ and $ x'=g'h'$ $(g \in G, h \in H)$, we have
\begin{gather}
\label{eqn:o1}
x <_{1} x' \iff 
\begin{cases}
ga(h) <_{\sf base} g'a(h') \text{, or,}\\
ga(h)=g'a(h') \text{ and } h <_{\sf base} (g^{-1}g')h'.
\end{cases}
\end{gather}

Note that by the proof of Lemma \ref{lemma:key0}, $ga(h)=g'a(h')$ implies $g^{-1}g' \in A$, hence $(g^{-1}g')h' \in \mF_{0.5}(X)$. Hence the inequality $h <_{\sf base} (g^{-1}g')h'$ makes sense.

\end{proof}

In a similar manner, we extend the ordering $<_{1}$ of $\mF_{1}(X)$ to a compatible ordering $<_{2}$ of $\mF_{2}(X)$. We define the map $c_{0}: \mF_{1}(X)-\mF_{0}(X) \rightarrow \mF_{0}(X)$ by 
\[ c_{0}(x)=\max_{<_{1}} \{ y \in \mF_{0}(X) \: | \: y <_{1} x\}, \]
and for $y \in \mF_{0}(X)$, we put 
\begin{eqnarray*}
\Delta_{y} & = &\{x \in \mF_{1}(X)-\mF_{0}(X) \: | \: c_{0}(x)=y\} \\
& = &\{x \in \mF_{1}(X) \: | \: y <_{1} x <_{1} yh_{\sf min}\}.
\end{eqnarray*}

\begin{lemma}
\label{lemma:c_0}
The map $c_{0}$ and the set $\Delta_{y}$ have the following properties.
\begin{enumerate}
\item For $x = gh \in \mF_{1}(X)-\mF_{0}(X)$ ($g \in G-A, h \in H$), $c_{0}(gh)=a(ga(h))h_{\sf M}$. 
Here $a: \mF_{0.5}(X) \rightarrow A$ is the map defined by (\ref{eqn:a}). 
\item For $x, x' \in \mF_{1}(X)-\mF_{0}(X)$ and $h,h' \in H$, if $hc_{0}(x) = h' c_{0}(x')$ then $h\Delta_{c_{0}(x)} = h'\Delta_{c_{0}(x')}$.
\end{enumerate}
\end{lemma}
\begin{proof}
Note that $a(ga(h)) <_{1} ga(h) <_{1} gh$. By definition of $<_{1}$ given in (\ref{eqn:o1}), there are no elements of $\mF_{0}(X)=H$ between $ga(h)$ and $gh$. Moreover, for $g \in G$ $c_{0}(g) = a(g)h_{\sf M}$ (see Figure \ref{fig:o1} again). This proves $c_{0}(gh) = c_{0}(ga(h))=a(ga(h))h_{\sf M}$. 

To see (2), write $x=gy$ and $x=g'y'$ $(g,g' \in G, \ y,y' \in\mF_{0})$. Then by (1), $hc_{0}(x)=h'c_{0}(x')$ implies that
$h^{-1}h' = c_{0}(x)c_{0}(x')^{-1} = a(ga(y))a(g'a(y'))^{-1} \in A$.
This shows $(h^{-1}h')\Delta_{c_{0}(x')} = \Delta_{c_{0}(x)}$ hence $h\Delta_{c_{0}(x)} = h'\Delta_{c_{0}(x')}$.
\end{proof}
 
\begin{proposition}
\label{prop:ext1to2}
There exists a unique compatible total ordering $<_{2}$ on $\mF_{2}(X)$ that extends $<_{1}$.
\end{proposition}
\begin{proof}

For $h \in H$ and $y \in \mF_{1}(X)$, we regard $h\Delta_{y}$ as a totally ordered set equipped with a total ordering $<_{2}$ defined by $hx <_{2}hx'$ $(x,x' \in \Delta_{y})$ if and only if $x <_{1} x'$. By the same argument as 
Proposition \ref{prop:extbto1}, this ordering is well-defined on each subset $h\Delta_{y}$.

$\mF_{2}(X) = \mF_{1}(X) \cup \left(\bigcup h\Delta_{y} \right)$ so we construct the desired ordering $<_{2}$ by inserting ordered set $h\Delta_{y}$ into $\mF_{1}(X)$, as we have done in Proposition \ref{prop:extbto1}.

By the compatibility requirement, for $x \in \Delta_{y}$ and $h \in H$, a desired extension $<_{2}$ must satisfy 
\[ hy <_{2} hx <_{2} hyh_{\sf min} \]
so we need to insert $h \Delta_{y}$ between $hc_{0}(x)$ and $h c_{0}(x) h_{\sf min}$. By Lemma \ref{lemma:c_0} (1), $\Delta_{y}$ is empty unless $y=ah_{\sf M}$ for some $a \in A$, and that if $\Delta_{y}$ is non-empty then $hy \prec_{1} hyh_{\sf min}$ for $h \in H-A$. That is, there are no elements of $\mF_{1}(X)$ between $hy$ and $hyh_{\sf min}$. Moreover, Lemma \ref{lemma:c_0} (2) shows that an ordered set $h\Delta_{y}$ inserted between $hy$ and $hyh_{\sf min}$ must be unique.

Thus, the process of inserting $h\Delta_{y}$ to $\mF_{1}(X)$ is unique, and we get a well-defined compatible ordering $<_{2}$. Figure \ref{fig:o2} gives schematic illustration of the inserting process.

\begin{figure}[htbp]
 \begin{center}
\includegraphics*[width=100mm]{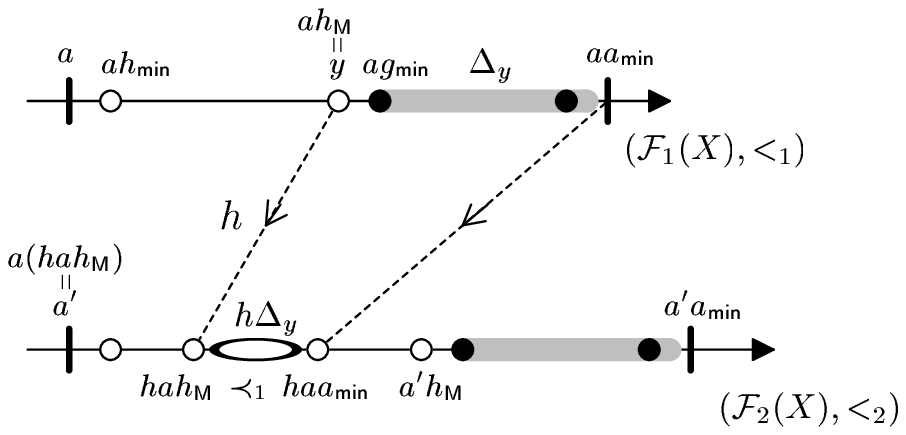}
\caption{Ordering $<_{2}$: Inserting $h\Delta_{y}$ between $hy=hah_{\sf M}$ and $hyh_{\sf min} = haa_{\sf min}$.}
\label{fig:o2}
\end{center}
\end{figure}

As a consequence, the ordering $<_{2}$ is given as follows:
For $x = hy$ and $ x'=h'y'$ $(h \in H, y \in \mF_{1}(X))$, we have
\begin{gather}
\label{eqn:o2}
x <_{2} x' \iff 
\begin{cases}
hc_{0}(y) <_{1} h'c_{0}(y') \text{, or,}\\
hc_{0}(y)=h'c_{0}(y') \text{ and } y <_{1} (h^{-1}h')y'.
\end{cases}
\end{gather}

Note that $hc_{0}(y)=h'c_{0}(y')$ implies $h^{-1}h' \in A$ as we have seen in the proof of Lemma \ref{lemma:c_0} (2), so the inequality $ y <_{1} (h^{-1}h')y' \in \mF_{1}(X)$ makes sense.

\end{proof}

Now we inductively extend compatible orderings. Assume that we have defined a compatible ordering $<_{i}$ of $\mF_{i+1}$. We define the map $c_{i-1}:\mF_{i}(X)-\mF_{i-1}(X) \rightarrow \mF_{i-1}(X)$ by
\[ c_{i-1}(x)=\max_{<_{i}} \{y \in \mF_{i-1}(X) \: | \: y <_{i} x \}\]
 and for $y \in \mF_{i-1}(X)$, we put 
\[ \Delta_{y} =\{ x \in \mF_{i}(X)-\mF_{i-1}(X) \: | \: c_{i-1}(x)=y\}. \]
Here we have assumed that $c_{i-1}$ is well-defined, that is, the maximal exists.

We will say that $<_{i}$ satisfies the \emph{Ping-Pong property} if the ordering $<_{i}$ satisfies the following three properties.

\begin{description}
\item[P1] The maps $c_{i-1}$ and $c_{i-2}$ satisfy the equality
\begin{gather*}
 c_{i-1}(x) = 
\begin{cases}
 gc_{i-2}(y) & (x=gy, g \in G-A, y \in \mF_{i-1}(X), \text{ if } i \text{ is odd})\\
 hc_{i-2}(y) & (x=hy, h \in H-A, y \in \mF_{i-1}(X), \text{ if } i \text{ is even}).
\end{cases}
\end{gather*}
Moreover, $c_{i-1}(x) \in \mF_{i-2}(X)-\mF_{i-3}(X)$.
\item[P2] $c_{i-1}(x) \prec_{i-1} c_{i-1}(x)h_{\sf min}$.
\item[P3] If $x \in \mF_{i}(X)-\mF_{i-2}(X)$, $x \prec_{i} xh_{\sf min}$.
\end{description}

The reason why we call these properties ``Ping-Pong'' will be explained in Remark \ref{rem:ping-pong}. Note that Ping-Pong property {\bf [P2]} shows that
\begin{equation}
\label{eqn:delta} \Delta_{y} =\{ x \in \mF_{i}(X)-\mF_{i-1}(X) \: | \: y <_{i} x <_{i} yh_{\sf min}\}. 
\end{equation}

\begin{lemma}
The ordering $<_2$ satisfies the Ping-Pong property.
\end{lemma}

\begin{proof}
This is easily seen from the description (\ref{eqn:o2}) of $<_{2}$ (see Figure \ref{fig:o2} again).

For $x=hy \in \mF_{2}(X)-\mF_{1}(X)$ ($h \in H-A$, $y \in \mF_{1}(X)-\mF_{0}(X)$), 
 $hc_{0}(y) <_{2} hy$. There are no elements of $\mF_{1}(X)-\mF_{0}(X)$ that lie between $hc_{0}(y)$ and $hy$ so $c_{1}(x)=hc_{0}(y)$. In particular, $c_{1}(x) \in \mF_{0}(X)=H$ hence by definition of $<_{1}$ given in (\ref{eqn:o1}) (see Figure \ref{fig:o1} again), $c_{1}(x) \prec_{1} c_{1}(x)h_{\sf min}$. Moreover, the description (\ref{eqn:o2}) of $<_{2}$ shows
\begin{gather*}
\begin{cases}
x \prec_{2} x h_{\sf min} & \text{if } x \not \in H \\
x \prec_{2} xh_{\sf M}^{-1}g_{\sf min} & \text{if } x \in H.
\end{cases}
\end{gather*}
\end{proof}

The Ping-Pong property shows the counterparts of Lemma \ref{lemma:key0} and  \ref{lemma:c_0}.
\begin{lemma}
\label{lemma:pp1}
Assume that $<_{i}$ satisfies the Ping-Pong property and let $x,x'\in \mF_{i}(X)-\mF_{i-1}(X)$. 
\begin{itemize}
\item If $i$ is odd, then $gc_{i-1}(x) = g'c_{i-1}(x')$ $(g,g' \in G)$ implies $g\Delta_{c_{i-1}(x)} = g'\Delta_{c_{i-1}(x')}$.
\item If $i$ is even, then $hc_{i-1}(x) = h'c_{i-1}(x')$ $(h,h' \in H)$ implies $h\Delta_{c_{i-1}(x)} = h'\Delta_{c_{i-1}(x')}$.
\end{itemize}
\end{lemma}
\begin{proof}
We show the case $i$ is odd. The case $i$ is even is similar.
Put $y= c_{i-1}(x)$ and $y'=c_{i-1}(x')$, respectively. 
We show $g' \Delta_{y'} \subset g \Delta_{y}$. The converse inclusion is proved similarly. By (\ref{eqn:delta}), $z' \in \Delta_{y'}$ if and only if $y' <_{i-1} z' <_{i-1} y'h_{\sf min}$. By compatibility,
\[ y = g^{-1}g'y' <_{i-1} (g^{-1}g')z' <_{i-1} g^{-1}g'y'h_{\sf min} = yh_{\sf min} \]
so $(g^{-1}g') z' \in \Delta_{y}$. This proves $g'z' \in g\Delta_{y}$.
\end{proof}

The following proposition completes the construction of isolated ordering $<_{X}$.

\begin{proposition}
\label{prop:proof}
If $<_{i}$ $(i>1)$ is a compatible ordering with the Ping-Pong property, then there exists a unique compatible ordering $<_{i+1}$ on $\mF_{i+1}(X)$ that extends $<_{i}$. Moreover, this compatible ordering $<_{i+1}$ also satisfies the Ping-Pong property.
\end{proposition}
\begin{proof}

The construction of $<_{i+1}$ is almost the same as the construction of $<_{2}$. We  treat the case $i$ is even. The case $i$ is odd is similar.

We regard each $g\Delta_{y}$ $(y \in \mF_{i-1}(X), g\in G-A)$ as a totally ordered set, by equipping a total ordering $<_{i+1}$ defined by $gx <_{i+1} gx'$ $(x,x' \in \Delta_{y})$ if and only if $x<_{i}x'$. By the same argument as Proposition \ref{prop:extbto1}, the ordering $<_{i+1}$ is well-defined on each $g\Delta_{y}$.
The desired compatible ordering $<_{i+1}$ on $\mF_{i+1}(X)=\mF_{i}(X) \cup \left( \bigcup g \Delta_{y}\right)$ is obtained by inserting $g\Delta_{y}$ into $\mF_{i}(X)$.

By the Ping-Pong property {\bf [P2]}, for $y \in \mF_{i-1}(X)$ if $\Delta_{y}$ is non-empty, then $y \prec_{i-1} yh_{\sf min}$. Thus we need to insert $g\Delta_{y}$ between $gy$ and $gyh_{\sf min}$. By the Ping-Pong property {\bf [P3]}, $gy \prec_{i} gyh_{\sf min}$, so there are no elements of $\mF_{i}(X)$ between $gy$ and $gyh_{\sf min}$. Moreover, Lemma \ref{lemma:pp1} shows that there are exactly one ordered set of the form $g\Delta_{y}$ that should be inserted between $gy$ and $gyh_{\sf min}$. Therefore the process of insertions is unique, and the resulting ordering $<_{i+1}$ is given as follows: For $x=gy$ and $ x'=g'y'$, ($g,g' \in G$ and $y, y' \in \mF_{i}(X)$), we define
\begin{gather}
\label{eqn:oi+1}
x <_{i+1} x' \iff 
\begin{cases}
gc_{i-1}(y) <_{i} g'c_{i-1}(y') \text{, or,}\\
gc_{i-1}(y)=g'c_{i-1}(y') \text{ and } y <_{i} (g^{-1}g')y'.
\end{cases}
\end{gather}

Next we show that $<_{i+1}$ also satisfies the Ping-Pong property. We have inserted $x=gy \in \mF_{i+1}(X)-\mF_{i}(X)$ $(g\in G-A, y \in \mF_{i}(X))$ between $gc_{i-1}(y)$ and $gc_{i-1}(y)h_{\sf min}$. By definition of $<_{i+1}$, there are no elements of $\mF_{i}(X)$ that lie between $x$ and $gc_{i-1}(x)$, hence $c_{i}(x)=gc_{i-1}(y)$.
By {\bf [P1]} for $<_{i}$, $c_{i-1}(y) \in \mF_{i-2}(X)-\mF_{i-3}(X)$. Hence $c_{i}(x) = gc_{i-1}(y) \in \mF_{i-1}(X)-\mF_{i-2}(X)$ so $<_{i+1}$ satisfies {\bf [P1]}.
Moreover by {\bf [P3]} for $<_{i}$, $c_{i}(x) \in \mF_{i-1}(X)-\mF_{i-2}(X)$ implies that $c_{i}(x) \prec_{i} c_{i}(x)h_{\sf min}$ hence $<_{i+1}$ satisfies {\bf [P2]}.

Finally we show that $<_{i+1}$ satisfies {\bf [P3]}. Assume that $x \in \mF_{i+1}(X)-\mF_{i}(X)$, and put $x=gy$ ($g \in G-A$, $y\in \mF_{i}(X)-\mF_{i-1}(X)$). By {\bf [P3]} for $<_{i}$, we have $y \prec_{i} yh_{\sf min}$. Hence by definition of $<_{i+1}$ we have $x = gy \prec_{i+1} gyh_{\sf min} = xh_{\sf \min}$. 

If $x \in \mF_{i}(X)-\mF_{i-1}(X) \subset \mF_{i}-\mF_{i-2}(X)$, then by {\bf [P3]} for $<_{i}$ we have $x \prec_{i} xh_{\sf min}$.
No elements of $\mF_{i+1}(X)-\mF_{i}(X)$ are inserted between $x$ and $xh_{\sf min}$, hence $x \prec_{i+1} xh_{\sf min}$. 

\end{proof}

\begin{proof}[Proof of Theorem \ref{theorem:main}]
For $x,x' \in X$, we define the isolated ordering $<_{X}=<_{X}^{(1)}$ by
\[ x <_{X} x' \iff x <_{N} x' \]
where $N$ is chosen to be sufficiently large so that $x,x' \in \mF_{N}(X)$.
Proposition \ref{prop:proof} shows that $<_{X}$ is a well-defined left ordering of $X$. By Proposition \ref{prop:charbase}, $<_{X}$ is isolated with characteristic positive set 
\[\{g_1,\ldots,g_m, h_1,\ldots,h_n, h_{\sf min}a_{\sf min}^{-1}g_{\sf min}\}, \]
if $\{g_1,\ldots,g_m\}$ is a characteristic positive set of $<_{G}$ and $\{ h_1,\ldots,h_n\}$ is a characteristic positive set of $<_{H}$ relative to $A$.

It remains to show that $A$ is a stepping with respect to $<_{X}$.
To see this, for $x \in X$, define
\[a(x) = a \circ \cdots \circ c_{N-2}\circ c_{N}(x) \]
where $N$ is taken so that $x \in \mF_{N}(X)$.
By definition of $c_{i}$, $a(x) = \max_{<_{X}}\{a \in A \: | \: a\leq_{X} x\}$.
\end{proof}

\begin{remark}
\label{rem:ping-pong}
Here we explain why we call the properties {\bf [P1]}--{\bf [P3]} the Ping-Pong property. This may help to understand isolated ordering $<_{X}$ we constructed.

Let us divide $X-A$ into two disjoint subsets $\mathcal{E}$ and $\mathcal{O}$ as follows:
\begin{gather*}
\begin{cases}
\mathcal{E} =\bigcup_{a \in A}\{x \in X\: | \: a <_{X} x <_{X} ah_{\sf M} \}\\
\mathcal{O} =\bigcup_{a \in A}\{x \in X\: | \: ag_{\sf min} <_{X} x <_{X} aa_{\sf min}.\}
\end{cases}
\end{gather*}
By definition of $<_{\sf base}$, $\mF_{0}(X)-A = H-A \subset \mathcal{E}$, and by definition of $<_{1}$, 
 $\mF_{1}(X)-\mF_{0}(X) = GH-H \subset \mathcal{O}$. Now the Ping-Pong property {\bf [P1]} says that 
\begin{gather*}
\begin{cases}
g(\mF_{2i}(X)-\mF_{2i-1}(X)) \subset \mathcal{E} & (g \in G-A) \\
 h(\mF_{2i+1}(X)-\mF_{2i-1}(X)) \subset \mathcal{O} & (h \in H-A) \\
\end{cases}
\end{gather*}
Thus, we conclude
\begin{gather*}
\begin{cases}
\mathcal{E}=\{\textrm{even part}\} = \bigcup_{i} (\mF_{2i}(X)-\mF_{2i-1}(X))\\
\mathcal{O}=\{\textrm{odd part}\} = \bigcup_{i} (\mF_{2i+1}(X)-\mF_{2i}(X))\\
\end{cases}
\end{gather*}
and for $g \in G-A$ and $h \in H-A$, we have
\[ g (\mathcal{O}) \subset \mathcal{E},\ \ h(\mathcal{E}) \subset \mathcal{O}. \] 

Therefore the subsets $\mathcal{O}$ and $\mathcal{E}$ provides the setting of a famous Ping-Pong argument. The rest of the Ping-Pong properties {\bf [P2],[P3]}, as we have seen in the proof of Proposition \ref{prop:proof}, rather follows from {\bf [P1]}. This explains why we call the properties {\bf [P1]--[P3]} the Ping-Pong property.
\end{remark}

\begin{remark}
\label{rem:computation}
Here we briefly explain the computability of the resulting isolated ordering $<_{X}$. 

By (\ref{eqn:oi+1}), for $x \in \mF_{i+1}(X)-\mF_{i}(X)$, determining whether $1 <_{X} x$ (which is equivalent to $ 1 <_{i+1} x$) is reduced to the computation of $c_{i}(x)$ and the ordering $<_{i}$. 
By Ping-Pong property {\bf [P1]}, $c_{i}(x)$ is computed from the function $c_{i-1}$. Thus, eventually one can reduce to the computations of the base orderings $<_{G}$ and $<_{H}$ and the map $a:\mF_{0.5}(X) \rightarrow A$. That is, we have:\\
\emph{The ordering $<_{X}$ is algorithmically computable if and only if the orderings $<_{G}$, $<_{H}$ and the map $a:\mF_{0.5}(X) \rightarrow A$ are algorithmically computable.}

The problem may occur when we want to compute the map $a$. Even if we have a nice algorithm to compute $<_{G}$ and $<_{H}$, this does not guarantee an algorithm to compute the map $a$, in general because it involves the maximum.
\end{remark}

Finally we study convex subgroups. Recall that a subset $C$ of a totally ordered set $(S,<_{S})$ is \emph{convex} if $c \leq_{S} s \leq_{S}c'$ $(c,c' \in C, s \in S)$ implies $s \in C$. For a subset $T$ of $(S,<_{S})$ the \emph{convex hull} $\Conv_{S}(T)$ of $T$ in $S$ is the minimum convex subset that contains $T$. Namely,
\begin{eqnarray*}
\Conv_{S}(T) & = & \bigcap_{\{C\supset T: \text{convex} \}} C\\
& = & \{s \in S \: | \: \exists t,t'\in T, t \leq_{S} s \leq_{S} t' \}.
\end{eqnarray*}

Let $(G,<_{G})$ be a left-ordered group and let $A$ be a subgroup of $G$. We denote the restriction of $<_{G}$ on $A$ by $<_{A}$. We say a convex subgroup $B$ of $(A,<_{A})$ is a \emph{$(G,<_{G})$-strongly convex} if $\Conv_{G}(B)$ is a subgroup of $G$.

\begin{proposition}
\label{prop:convex}
Let $<_{X}$ be an isolated ordering on $X=G*_{A}H$ in Theorem \ref{theorem:main}.
If a subgroup $B$ of $A$ is both $(G,<_{G})$- and $(H,<_{H})$-strongly convex, then $B$ is $(X,<_{X})$-strongly convex. In particular, if $B$ and $B'$ are different then $\Conv_{X}(B)$ and $\Conv_{X}(B')$ yield different convex subgroups of $(X,<_{X})$.
\end{proposition}
\begin{proof}

The case $B=\{1\}$ is trivial so we assume that $B \neq \{1\}$.
By induction on $N$, we prove that if $x \in \Conv_{X}(B) \cap \mF_{N}(X)$ then $xx' \in \Conv_{X}(B)$ for any $x' \in \Conv_{X}(B)$.

First assume that $x \in \mF_{0.5}(X) = G\cup H$.
For $x' \in \Conv_{X}(B)$, take $b \in B$ so that $b^{-1} <_{X} x' <_{X} b$. 
Then $xb^{-1} <_{X} xx' <_{X} xb$. Since $B$ is $(G,<_{G})$- and $(H,<_{H})$-strongly convex, $xb,xb^{-1} \in \Conv_{G}(B) \cup \Conv_{H}(B) \subset \Conv_{X}(B)$, hence $xx' \in \Conv_{X}(B)$.

To show general case, assume that $x \in \mF_{N}(X)-\mF_{N-1}(X)$ and put $x = gy$ $(g \in G-A, y \in \mF_{N-1}(X))$. We consider the case $N$ is odd, since the case $N$ is even is similar.

By Theorem \ref{theorem:main} (3), $A$ is a stepping so 
\[ a(y)= \max_{<_{X}}\{a \in A \: |\: a<_{X} y\} \]
exists. On the other hand, $x \in \Conv_{X}(B)$ so there exists $b \in B \subset A$ such that $b^{-1} <_{X} x <_{X} b$. By definition of $a(y)$, 
\[ b^{-1} \leq_{X} g a(y) <_{X} x <_{X} b\]
hence $ga(y) \in \Conv_{X}(B)$. We have assumed that $B$ is a non-trivial convex subgroup of $A$, so $a_{\sf min} \in B$. Since $1 <_{X} a(y)^{-1}y <_{X} a_{\sf min} $, $a(y)^{-1}y \in \Conv_{X}(B)$. By induction, $(a(y)^{-1}y)x' \in \Conv_{X}(B)$ if $x' \in \Conv_{X}(B)$. This shows that 
\[ x x'= (ga(y))(a(y)^{-1}y)x' \in \Conv_{X}(B) \]
as desired.
\end{proof}

We close the paper by giving new examples of isolated orderings obtained by Theorem \ref{theorem:main}.

\begin{example}
\label{exam:convex}
Let $B_{3}$ be the 3-strand braid group, given by 
\begin{eqnarray*}
 B_{3} & =  & \Z*_{\Z} \Z = \langle x,y \: | \: x^{2}=y^{3} \rangle\\
 & = & \left \langle 
\sigma_{1},\sigma_{2} \: | \: \sigma_{1}\sigma_{2}\sigma_{1}=\sigma_{2}\sigma_{1}\sigma_{2} \right\rangle .
\end{eqnarray*}
By Theorem \ref{theorem:main}, $B_{3}$ admits an isolated ordering $<_{DD}$, which is known as the Duborvina-Dubrovin ordering \cite{dd}.
The Dubrovina-Dubrovin ordering is discrete with minimum positive element $\sigma_{2}$. For $p>1$, let $A=A_p$ be the kernel of the mod $p$ abelianization map $e:B_{n} \rightarrow \Z_{p}$. Since for  $x \in B_{3}$
\[ \cdots \prec_{DD} x \sigma_{2}^{-1} \prec_{DD} x \prec_{DD} x \sigma_{2} \prec_{DD} x \sigma_{2}^{2} \prec_{DD} \cdots, \]
$A$ is a stepping with respect to $<_{DD}$: The maximum and minimum functions are given by
\[ a(x)=x \sigma_{2}^{-e(x)},\ \  a_{+}(x) = x \sigma_{2}^{p-e(x)} \ \ ( e(x) \in \{0,1,\ldots,p-1\}). \]

By Theorem \ref{theorem:main}, $X=X_p=B_{3}*_{A_p}B_{3}$ admits an isolated ordering $<_{X}$. The convex subgroup $B$ of $A$ generated by $\sigma_{2}^{p}$ is $(B_{3},<_{DD})$-strongly convex, hence by Proposition \ref{prop:convex}, $\Conv_{X}(B)$ is a non-trivial proper convex subgroup of $(X,<_{X})$.
$(X,<_{X})$ contains another non-trivial proper convex subgroup generated by the $<_{X}$-minimum positive elements, so $(X,<_{X})$ has at least two non-trivial proper convex subgroup. 
Iterating this kinds of constructions, starting from $\Z$ we are able to construct isolated ordering with arbitrary many proper non-trivial convex subgroups.
\end{example}


\begin{thebibliography}{1}

\bibitem{bg} V. Bludov and A. Glass,
{\em{Word problems, embedddings, and free products of right-ordered groups with amalgamated subgroups,}} 
Proc. London. Math. Soc. \textbf{99} (2009) 585--608.

\bibitem{ddrw} P. Dehornoy, I.Dynnikov, D.Rolfsen and B.Wiest, 
{\em{Ordering Braids,}}
Mathematical Surveys and Monographs \textbf{148}, Amer. Math. Soc. 2008.
\bibitem{de} P. Dehornoy,
{\em{Monoids of O-type, subword reversing, and ordered groups,}}
J. Group Theory. to appear.
\bibitem{dd} T. Dubrovina and T. Dubrovin,
{\em{On braid groups,}}
Sb. Math, \textbf{192} (2001), 693--703.

\bibitem{i1} T. Ito, 
{\em{Dehornoy-like left orderings and isolated left orderings,}}
J. Algebra \textbf{374} (2013), 42--58.
\bibitem{i2} T. Ito,
{\em{Construction of isolated left orderings via partially central cyclic amalgamation}},
arXiv:1107.0545.
\bibitem{km} V. Kopytov and N. Medvedev,
{\em{Right-ordered groups,}}
Siberian School of Algebra and Logic, Consultants Bureau, 1996.


\bibitem{n1} A. Navas,
{\em{On the dynamics of (left) orderable groups,}}
Ann. Inst. Fourier,  \textbf{60}  (2010), 1685--1740.

\bibitem{n2} A. Navas,
{\em{A remarkable family of left-ordered groups: Central extensions of Hecke groups,}}
J. Algebra, \textbf{328}  (2011), 31--42.

\bibitem{ri} C. Rivas,
{\em{Left-orderings on free products of groups,}}
J. Algebra \textbf{350} (2012), 318--329.

\bibitem{rt} C. Rivas and R. Tessera,
{\em{On the space of left-orderings of virtually solvable groups,}}
arXiv:1209.3251

\bibitem{si} A. Sikora,
{\em{Topology on the spaces of orderings of groups,}}
Bull. London Math. Soc. \textbf{36}, (2004), 519-526.

\end{thebibliography}
\end{document}